\documentclass{article}
\usepackage{geometry}
\usepackage{titling}
\AtEndDocument{\bigskip{\footnotesize%
\textsc{Department of Applied Mathematics, Delhi Technological University, Delhi–110042, India} \par  
  \textit{E-mail address:} \texttt{suryagiri456@gmail.com} \par
  \addvspace{\medskipamount}
  \textsc{Department of Applied Mathematics, Delhi Technological University, Delhi–110042, India} \par
  \textit{E-mail address:} \texttt{spkumar@dtu.ac.in} \par
}}
\usepackage{doc}

\usepackage{url}

\usepackage{graphicx}
\usepackage{epstopdf}
\usepackage{amsthm}
\usepackage{amssymb}
\usepackage{amsmath}

\DeclareMathOperator{\RE}{Re}
\usepackage{hyperref}
\usepackage{array}
\usepackage{enumerate}
\usepackage{enumitem}
\usepackage[utf8]{inputenc}
\usepackage[english]{babel}
\newtheorem{theorem}{Theorem}
\newtheorem{corollary}{Corollary}[theorem]
\newtheorem{lemma}[theorem]{Lemma}

\newtheorem*{remark}{Remark}

\begin{document}

%%
%% The title of the paper goes here.  Edit to your title.
%%

\title{Hermitian-Toeplitz Determinants for Certain Univalent Functions}
\author{Surya Giri and S. Sivaprasad Kumar}
%\address{Department of Applied Mathematics, Delhi Technological University, Delhi, SC 29208}
%\email{}
%\author{Surya Giri Mahant}
%\address{}
%\email{}
\date{}

\maketitle

%%
%% Now edit the following to give your name and address:
%% If there is another author uncomment and edit the following.
%%

%\author{Second Author}
%\address{Department of Mathematics, University of South Carolina,
%Columbia, SC 29208}
%\email{second@math.sc.edu}
%\urladdr{www.math.sc.edu/$\sim$second}

%%
%% If there are three of more authors they are added in the obvious
%% way

%%%
%%% The following is for the abstract.  The abstract is optional and
%%% if not used just delete, or comment out, the following.
%%%

\begin{abstract}
    Sharp upper and lower bounds for the second and third order Hermitian-Toeplitz determinants are obtained for some generalized subclasses of starlike and convex functions. Applications of these results are also discussed for several widely known classes.
\end{abstract}
\vspace{0.5cm}
	\noindent \textit{Keywords:} Univalent functions; Starlike functions; Close-to-Convex functions; Hermitian-Toeplitz Determinant.\\
	\noindent \textit{AMS Subject Classification:} 30C45, 30C50, 30C80.

%%
%%  LaTeX will not make the title for the paper unless told to do so.
%%  This is done by uncommenting the following.
%%

%\maketitle

%%
%% LaTeX can automatically make a table of contents.  This is done by
%% uncommenting the following:
%%

%\tableofcontents

%%
%%  To enter text is easy.  Just type it.  A blank line starts a new
%%  paragraph.
%

\section{Introduction}
    Let $\mathcal{A}$ be the class of functions of the form $f(z)=z+a_2 z^2 +a_3 z^3  +\cdots$, $(a_2\neq 0)$, which are analytic in the open unit disk $\mathbb{D}=\left\{ z\in \mathbb{C}:\lvert z \rvert < 1\right\}$ and $\mathcal{S}$ be the subclass of $\mathcal{A}$, consisting of univalent functions. The subclasses of $\mathcal{S}$ of starlike and convex functions are denoted by $\mathcal{S}^*$ and $\mathcal{C}$ respectively. A function $f\in \mathcal{S}^*$ if and only if
    $\RE \left( {z f'(z)}/{f(z)} \right)>0,$ $z\in \mathbb{D}$.
   Also, a function $f \in \mathcal{C}$ if and only if
   $\RE  \left(1 + {z f''(z)}/ {f'(z)} \right) >0,$ $z\in \mathbb{D}$.
   A function $f\in \mathcal{A}$ is said to be close-to-convex \cite{kaplan} if there is a function $g\in \mathcal{S}^*$ such that
\begin{equation}\label{clcv}
   \RE \left( \frac{z f'(z)}{g(z)} \right)>0.
\end{equation}
   Geometrically, a function $f$ is said to be close-to-convex if the image of $\lvert z\rvert=r <1$ has no ``large hairpin'' turns; that is, there are no sections of the image curve under $f$ in which the tangent vectors turn backward through an angle greater than or equal to $\pi.$ We denote by $\mathcal{K}$ the class of close-to-convex functions.
   Various choices of $g(z)$ yield some well known subclasses of $\mathcal{K}$. For instance,
\begin{align*}
     \mathcal{F}_1 &=\left\{ f\in \mathcal{S}: \RE (1-z)f'(z)>0\right\},\\
     \mathcal{F}_2 &=\left\{ f\in \mathcal{S}: \RE (1-z^2)f'(z)>0\right\},\\
     \mathcal{F}_3 &=\left\{ f\in \mathcal{S}: \RE (1-z)^2f'(z)>0\right\},\\
      \mathcal{F}_4 &=\left\{ f\in \mathcal{S}: \RE (1-z+z^2)f'(z)>0\right\}
\end{align*}
  and
    $$ \mathcal{R} =\left\{ f\in \mathcal{S}: \RE f'(z)>0\right\}. $$
    Ozaki \cite{ozaki} has proved that condition (\ref{clcv}) is sufficient  for a function $f$ to be univalent. Functions in the class $\mathcal{F}_2$ and $\mathcal{F}_4$ exhibit nice geometric properties.  Functions in the class $\mathcal{F}_2$ map $\mathbb{D}$ univalently onto a domain convex in the direction of imaginary axis whereas functions in the class $\mathcal{F}_4$ map $\mathbb{D}$ univalently onto a domain convex in the direction of real axis.

    For an analytic univalent function $\varphi(z)$ having positive real part with $\varphi(0)=1$, $\varphi'(0)>0$,  starlike with respect to 1 and symmetric about the real axis,  Ma and Minda \cite{MaMinda} introduced the following subclasses:
    $$ \mathcal{S}^*(\varphi)=\left\{ f \in \mathcal{A}: \frac{z f'(z)}{f(z)}\prec \varphi(z) \right\} ,$$
    $$ \mathcal{C}(\varphi)=\left\{ f \in \mathcal{A}: 1+ \frac{z f''(z)}{f'(z)}\prec \varphi(z) \right\}, $$
    where the symbol `$\prec$' is used for subordination \cite{good}.
   On taking $\varphi(z)= 1 + \left(2/{\pi^2}\right) \left(\log\left((1 + \sqrt{z})/(1 - \sqrt{z})\right)\right)^2$, we obtain the class $\mathcal{S}_P$ of parabolic starlike functions introduced by R\o nning \cite{ronning}. Various choices of $\phi(z)$ yield different subclasses of starlike and convex functions. For instance,
    when $\varphi(z)=(1+A z)/(1+ B z)$, where $-1 \leq B < A \leq 1$, we obtain the classes of Janowski starlike functions $\mathcal{S}^*[A,B]$ and Janowski convex functions $\mathcal{C}[A,B]$ \cite{Jano}. For $A=1-2\alpha$ and $B=-1$, where $0\leq \alpha<1$, classes $\mathcal{S}^*[A,B]$ and $\mathcal{C}[A,B]$ reduce to the classes of  starlike functions of order $\alpha$, $\mathcal{S}^*(\alpha)$ and convex functions of order $\alpha$, $\mathcal{C}(\alpha)$ respectively \cite{good}. Classes $\mathcal{S}^*(\varphi)$ and $\mathcal{C}(\varphi)$ reduce to the class of strongly starlike functions $\mathcal{SS}^*(\alpha)$ and strongly convex functions $\mathcal{CC}(\alpha)$ for $\varphi(z)=((1+z)/(1-z))^{\alpha}$  \cite{good}. Various subclasses of $\mathcal{S}^*$ and $\mathcal{C}$ are considered for different choices of $\varphi(z)$ \cite{Goel,kamal}. For other subclasses and their corresponding class notation see \cite[Table 1]{cft}. We use these notations directly in the results.

    For the coefficient $\left\{a_k\right\}_{k\geq 2}$ of functions $f\in \mathcal{A}$, the Toeplitz matrix  is defined by  (see \cite{Hartman})
\begin{equation}\label{intil}
     T_{m,n}(f)= \begin{bmatrix}
	a_n & a_{n+1} & \cdots & a_{n+m-1} \\
	\bar{a}_{n+1} & a_n & \cdots & a_{n+m-2}\\
	\vdots & \vdots & \vdots & \vdots\\
    \bar{a}_{n+m-1} & \bar{a}_{n+m-2} & \cdots & a_n\\
	\end{bmatrix}.
\end{equation}
     The Toeplitz matrix $T_{m,n}(f)$ is Hermitian-Toeplitz matrix when $a_n$ is real as $T_{m,n}(f)=\overline{(T_{m,n}(f))^T}$ and determinants of Hermitian matrices are real numbers.  It is also noted that the determinant of $T_{m,1}(f)$ is rotationally invariant that is determinant of $T_{m,1}(f)$ and $T_{m,1}(f_\theta)$ are same, where $f_\theta=e^{-i \theta}f(e^{i\theta }z)$ and $\theta\in \mathbb{R}$. Since for $n=1$ and $f\in \mathcal{A}$, $a_1=1$. Thus,
   the second order Hermitian-Toeplitz determinant is
  $$ \det{T_{2,1}(f)}=1-\lvert a_2 \rvert^2$$
   and the third order Hermitian-Toeplitz determinant is given by
\begin{equation}\label{T31}
    \det T_{3,1}(f)=
     \begin{vmatrix}
     1 & a_{2} & a_{3} \\
	 \bar{a}_{2} & 1 & a_{2}\\
	 \bar{a}_3 & \bar{a}_{2} & 1\\
     \end{vmatrix}
      = 2 \RE \left(a_2^{2} \bar{a}_3\right)- 2\lvert a_2\rvert^2- \lvert a_3\rvert^2+1.
\end{equation}
   Toeplitz and Hankel matrices are closely related to each other. Toeplitz matrices contain constant entries along their diagonals whereas Hankel matrices contain constant entries along their reverse diagonals. Ye and Lim \cite{LHLIM} showed that any $n \times n$ matrix over $\mathbb{C}$ generically can be written as the product of some Toeplitz matrices or Hankel matrices. Toeplitz matrices and Toeplitz determinants have numerous applications in the field of pure as well as applied mathematics. They arise in partial differential equations, algebra, signal processing and time series analysis. For more applications of Toeplitz matrices and Toeplitz determinants, we refer \cite{LHLIM} and the references cited therein.

   Finding the upper and lower bound of Hermitian-Toeplitz determinants for initial values of $m$ and $n$ is now in trend in GFT. Cudna et al. \cite{cudna} determined sharp lower and upper estimate for $\det T_{2,1}(f)$ and $\det T_{3,1}(f)$ for the functions belonging to the classes of starlike and convex functions of order $\alpha$, $0\leq \alpha<1$. Kumar et al. \cite{Vkumar2} generalized this work  and obtained the bound for the classes of Janowski starlike and Janowski convex functions. Obradovi\'{c} and Tuneski \cite{tuneski} derived the bounds for the class $\mathcal{S}$ and some of its subclasses.
    Recently, Kumar \cite{Vkumar3} obtained the sharp lower and upper bound of $\det{T}_{2,1}(f)$ and $\det{T}_{3,1}(f)$ for the classes $\mathcal{F}_1$, $\mathcal{F}_2$, $\mathcal{F}_3$ and $\mathcal{F}_4$. Kowalczyk\ et al. \cite{B.Kowa} also attained the bounds for the classes $\mathcal{F}_2$ and $\mathcal{F}_3$. For more work in this direction one can see \cite{ahuja,MFALI,Cho,Lecko,Lecko2}.

    For different choices of $\varphi$ in $\mathcal{S}^*(\varphi)$ and $\mathcal{C}{(\varphi)}$, many authors obtained upper as well as lower bounds of Hermitian-Toeplitz determinants whereas for general $\varphi$ it is still open. Here, we solve this problem. Further, in case of the class $\mathcal{K}$ bounds of $\det{T_{2,1}(f)}$ and $\det{T_{3,1}(f)}$ are established for certain $g \in \mathcal{S}^*.$
%%%%%%%%%%%%%%%%%%%%%%%%%%%%%%%%%%%%%%%%%%%%%%%%%%%%%%%%%%%%%%%%%%%%%%%%%%%%%%%%%%%%%%
   We use the following lemma as preliminaries.
%\begin{lemma}\label{cst}\cite[Vol. I,p.128]{good}
  %  If $f(z)=z+\sum_{n=2}^\infty a_n z^n$ and
  %  $$\sum_{n=2}^\infty n \lvert a_n \rvert\leq 1, $$
 %   then $f\in \mathcal{S}^*.$
%\end{lemma}
%%%%%%%%%%%%%%%%%%%%%%%%%%%%%%%%%%%%%%%%%%%%%%%%%%%%%%%%%%%%%%%%%%%%%%%%%%%%%%%%%%%%%%%%%%%%%%%%%%%%%
\begin{lemma}\label{lemma1}\cite[Lemma 3]{Libera}
    If $p(z)=1 +\sum_{n=1}^\infty p_n z^n \in \mathcal{P}$, then
    $$ 2 p_2=p_1^2  +  \left(4 -p_1^2\right)\zeta,$$
    for some $\zeta\in \mathbb{\bar{D}}.$
\end{lemma}
\section{Main Results}
     In the following result, we obtain the bounds of $\det{T_{2,1}(f)}$ and $\det{ T_{3,1}(f)}$ for the class $\mathcal{S}^*(\varphi)$.
\begin{theorem}\label{TUB}
     Let $f\in \mathcal{S}^*(\varphi)$ and $\varphi(z)=1+B_1 z+ B_2 z^2 + B_3 z^3 +\cdots$, then the following holds:
\begin{enumerate}
\item   $1- B_1^2 \leq \det{T_{2,1}(f)}\leq 1 .$
\item   If $3 B_1^4  - 8 B_1^2  + 2 B_1^2 B_2 - B_2^2  < 0$ and $B_1 \leq \lvert B_2 + B_1^2 \rvert$, then
\begin{equation}\label{ninq}
    \det{ T_{3,1}(f)} \leq 1.
\end{equation}
\item If $3 B_1^4  - 8 B_1^2  + 2 B_1^2 B_2 - B_2^2 \geq 0$ and $B_1 \leq \lvert B_2 + B_1^2 \rvert$, then
\begin{equation}\label{UB}
       \det{ T_{3,1}(f)} \leq B_1^2(B_1^2 +B_2) -\frac{1}{4}\left( B_1^2+B_2 \right)^2 -2 B_1^2 + 1.
\end{equation}
\end{enumerate}
    All these inequalities are sharp.
\end{theorem}
\begin{proof}
    Let $f(z)= z+ \sum_{n=2}^\infty a_n z^n \in \mathcal{S}^*(\varphi)$, then $\lvert a_2\rvert\leq B_1$ and whenever $B_1 \leq \lvert B_2 + B_1^2 \rvert$, we have
    $$ \lvert a_3\rvert\leq \frac{1}{2}\lvert  B_1^2+B_2\rvert$$
    (see \cite{ahuja}). It is clear that, the lower and the upper estimates of $\det{T_{2,1}(f)}$  will be deduced directly from the bound of $\lvert a_2\rvert$, which are sharp for the functions $f_1(z)$ and $f_2(z)$ respectively, given by
\begin{equation}\label{extremalf1}
     f_1(z)=z\exp\int_0^z \frac{\varphi(t)-1}{t} dt \;\;  \text{and} \; \;   f_2(z)=z\exp\int_0^z \frac{\varphi(t^2)-1}{t} dt.
\end{equation}
    Now we proceed to estimate $\det{T_{3,1}(f)}$. Using the inequality $\RE (a_2^2 \bar{a}_3 )\leq \lvert a_2\rvert^2 \lvert a_3\rvert$ in (\ref{T31}), we have
    $$\det{T_{3,1}(f)}\leq 2 \lvert a_2\rvert^2 \lvert a_3\rvert- 2 \lvert a_2 \rvert^2-\lvert a_3\rvert^2+1=:u(\lvert a_3\rvert),$$
    where $u(x)=2  \lvert a_2\rvert^2 x- 2 \lvert a_2 \rvert^2- x^2 +1.$ Clearly, maximum of $u(x)$ is the upper bound for $\det{T_{3,1}(f)}$. Since $B_1\in [0,2]$ and $B_2\in [-2,2]$, we have $x=\lvert a_3\rvert\in [0,3]$. Since,
    $$u'(x)=2 \left( \lvert a_2\rvert^2 -x \right)\; \;  \text{and} \; \; u''(x)=-2,$$
     $u(x)$ attains its maximum value at $x_0=\lvert a_2\rvert^2$. Now, there arises two cases;\\
    {\bf{Case 1:}} when $\lvert a_2 \rvert^2$ lies in the range of $x$ that is $\lvert a_2\rvert^2 <  \lvert a_3 \rvert $,
%    then
%    $$0\leq \lvert a_2\rvert^2 \leq \frac{1}{2}\lvert B_1^2+B_2\rvert,$$
     then
\begin{align*}
    \max{u(x)}=u(\lvert a_2\rvert^2)&=\left(\lvert a_2\rvert^2-1 \right)^2 \\
                                    & \leq \left\{
                                    \begin{array}{ll}
                                       1  &  \lvert a_2\rvert^2\leq 2,\\
                                       \left(\frac{1}{2}\left(B_1^2 +B_2 \right) -1\right)^2   &2\leq \lvert a_2\rvert^2 \leq \frac{1}{2}\lvert B_1^2+B_2\rvert.
                                    \end{array}
                                      \right. \\
                                    &= \left\{
                                    \begin{array}{ll}
                                       1  &  \lvert B_1^2+B_2\rvert \leq 4,\\
                                       \left(\frac{1}{2}\left(B_1^2 +B_2 \right) -1\right)^2   &4 \leq \lvert B_1^2+B_2\rvert.
                                    \end{array}
                                      \right.
\end{align*}
{\bf{Case 2:}} If $\lvert a_3 \rvert \leq \lvert a_2\rvert^2 ,$ %that is $ \lvert B_1^2+B_2\rvert >4 $
    then
  %  Second, when
  % $$\frac{1}{2}\lvert B_1^2+B_2\rvert \leq \lvert a_2\rvert^2 \leq B_1^2,$$
  % then
\begin{align*}
            \max{u(x)}&=u(\lvert a_3\rvert)\\
                                   &= u\left(\frac{1}{2}\lvert B_1^2 +B_2\rvert \right)\leq  B_1^2(B_1^2 +B_2) -\frac{1}{4}\left( B_1^2+B_2 \right)^2 -2 B_1^2 + 1.
\end{align*}
   It can be easily seen that for $\lvert B_1^2 +B_2 \rvert \geq 4$
\begin{align*}
   \max\bigg\{\left(\frac{1}{2}\left(B_1^2 +B_2 \right) -1\right)^2&,  B_1^2(B_1^2 +B_2) -\frac{1}{4}\left( B_1^2+B_2 \right)^2 -2 B_1^2 + 1\bigg\} \\
                                                                               &=   B_1^2(B_1^2 +B_2) -\frac{1}{4}\left( B_1^2+B_2 \right)^2 -2 B_1^2 + 1.
\end{align*}
   Also,
   $$ B_1^2(B_1^2 +B_2) -\frac{1}{4}\left( B_1^2+B_2 \right)^2 -2 B_1^2 + 1\geq 1 $$
   whenever
   $$  3 B_1^4  - 8 B_1^2  + 2 B_1^2 B_2 - B_2^2  \geq 0.$$
   Combining all these facts, we obtain the upper bound of $\det{T_{3,1}(f)}$.

   \textbf{Sharpness:}
 When $3 B_1^4  - 8 B_1^2  + 2 B_1^2 B_2 - B_2^2  < 0$, the inequality (\ref{ninq}) is sharp for  $f(z)=z$ and when $3 B_1^4  - 8 B_1^2  + 2 B_1^2 B_2 - B_2^2  \geq 0$, the equality sign in (\ref{UB}) holds for the function $f_1(z)$ given by (\ref{extremalf1}).
\end{proof}
     For $\varphi(z)=1+\sin{z}$, the class $\mathcal{S}^*(\varphi)$ reduces to the class $\mathcal{S}^*_{sin}$. From Theorem \ref{TUB}, we easily obtain the following result for the class $\mathcal{S}^*_{sin}$. Similarly, results of other classes are also derived.
\begin{corollary}
\begin{enumerate}
  \item If $f \in \mathcal{S}^*_{sin}$, then $ 0 \leq \det{T_{2,1}(f)}\leq 1$ and $\det{T_{3,1}(f)}\leq 1$.
  \item If $f \in \mathcal{S}^*_\varrho$, then $ 0 \leq \det{T_{2,1}(f)}\leq 1$  and $\det{T_{3,1}(f)}\leq 1$.
   \item If  $f \in \mathcal{S}_P$, then $1-(64/\pi^4) \leq \det{T_{2,1}(f)}\leq 1$ and $\det{T_{3,1}(f)}\leq 1.$
\end{enumerate}
\end{corollary}
\begin{remark}   
   Theorem \ref{TUB} yields some already known results for different subclasses of $\mathcal{S}^*$, obtained by an appropriate choice of $\varphi.$
\begin{enumerate}
    \item If $f\in \mathcal{S}^*_{SG}$, then $ 3/4 \leq \det{T_{2,1}(f)}\leq 1$ and $\det{T_{3,1}(f)}\leq 1$ \cite[Theorem 2.1]{Cho}.
    \item If $f\in \mathcal{S}^*_{B}$, then $ 0 \leq \det{T_{2,1}(f)}\leq 1$ and $\det{T_{3,1}(f)}\leq 1$ \cite[Theorem 2.2]{Cho}.
    \item If $f\in \Delta^*$, then $\det{T_{3,1}(f)}\leq 1$ \cite[Theorem 2]{Vkumar}.
    %\item If $f\in \mathcal{S}^*_L$, then $\det{T_{3,1}(f)}\leq 1$ \cite[Theorem 3]{Vkumar}.
    \item If $f\in \mathcal{S}^*$, then $\det{T_{3,1}(f)}\leq 8$ \cite[Corollary 3]{cudna}, \cite[Corollary 2]{B.Kowa}.
    \item If  $f\in \mathcal{S}^*(1/2)$, then $\det{T_{3,1}(f)}\leq 1$ \cite[Corollary 4]{cudna}.
    \item If  $f\in \mathcal{S}^*_{Ne}$, then $\det{T_{3,1}(f)}\leq 1$ \cite[Theorem 4.2]{Sushil}.
    \item If we take $\varphi(z)=(1+(1-2 \alpha)z)/(1-z), \alpha\in (0,1]$, then we obtain the upper bound of $\det{T}_{3,1}(f)$ for $f\in \mathcal{S}^*(\alpha)$ \cite[Theorem 3]{cudna}.
    \item For $\varphi(z)=((1+z)/(1-z))^\alpha, \alpha\in [1/3,1]$, we get the bound of $T_{3,1}(f)$ for function $f$ belonging to the class of strongly starlike function $\mathcal{SS}^*(\alpha)$ \cite[Theorem 1]{Vkumar}, \cite[Theorem 3]{B.Kowa}.
    \item For $\varphi(z)=(1+A z)/(1+B z)$, where $-1\leq B< A \leq 1$ with $A - B \leq \lvert A^2 - 3 A B + 2 B^2 \rvert $, we get the bound of $T_{3,1}(f)$ for the class $\mathcal{S}^*[A,B]$ \cite[Theorem 2]{Vkumar2}.
\end{enumerate}
\end{remark}
  In a similar fashion, the bounds of $\det{T_{2,1}(f)}$ can also be found for all the above mentioned classes.
\begin{theorem}\label{TKUB}
    If $f\in \mathcal{C}(\varphi)$ and $\varphi(z)=1+B_1 z+ B_2 z^2 +B_3 z^3+\cdots$, then the following hold:
\begin{enumerate}
  \item  $ 1- \frac{B_1^2}{4} \leq \det{T_{2,1}(f)}\leq 1 .$
  \item If $B_1 \leq \lvert B_2 + B_1^2 \rvert$, then
\begin{equation}\label{KUB}
      \det{T_{3,1}(f)} \leq 1.
\end{equation}
\end{enumerate}
   All these estimates are sharp.
\end{theorem}
\begin{proof} Let  $f(z)= z+ \sum_{n=2}^\infty a_n z^n \in \mathcal{C}(\varphi)$, then $\lvert a_2\rvert\leq {B_1}/{2}$ and whenever $B_1 \leq \lvert B_2 + B_1^2 \rvert$ hold, we have
     $$ \lvert a_3\rvert\leq \frac{1}{6}\lvert B_1^2+B_2\rvert$$
     (see \cite{ahuja}). The bounds of $\det{T_{2,1}(f)}$ can be obtained using the bound of $\lvert a_2 \rvert$.
     Further, the functions $f_3$ and $f_4$ satisfying
\begin{equation*}\label{ExtC}
       1+\frac{z f_3''(z)}{f_3'(z)}=\varphi(z) \;\;  \text{and} \; \;   1+\frac{z f_4''(z)}{f_4'(z)}=\varphi(z^2)
\end{equation*}
     respectively, act as extremal functions for lower and upper estimates of $\det{T_{2,1}(f)}$.
     Now using the technique of Theorem \ref{TUB}, in context of the class $\mathcal{C}(\varphi)$, we get
      $$\det{T_{3,1}(f)}\leq u(x)=2  \lvert a_2\rvert^2 x- 2 \lvert a_2 \rvert^2- x^2 +1,$$
      where $x=\lvert a_3\rvert$.  Clearly, $u(x)$ attain its maximum value at $x=\lvert a_2\rvert^2$. Therefore,
\begin{align*}
    \max u(x)= g(\lvert a_2\rvert^2)
    =\left( \lvert a_2\rvert^2 -1 \right)^2
                                    \leq 1.
\end{align*}
    The bound is sharp for the identity function $f(z)=z$.
\end{proof}
     Theorem \ref{TKUB} establishes known results for different subclasses of $\mathcal{C}$ based on the choice of $\varphi$.
\begin{remark}
\begin{enumerate}
    \item If $\varphi(z)=(1+z)/(1-z)$, then $f\in \mathcal{C}$ and $\det{T_{3,1}(f)}\leq 1$ \cite[Theorem 1]{Lecko}.
    \item If $\varphi(z)=(1+(1-2\alpha)z)/(1-z)$, then $f\in \mathcal{C}(\alpha)$ and $\det{T_{3,1}(f)}\leq 1$ \cite[Theorem 5]{cudna}.
    \item If $\varphi(z)=((1+z)/(1-z))^{\alpha}$, $\alpha \in [1/3,1]$, then $f\in \mathcal{CC}(\alpha)$ and $\det{T_{3,1}(f)}\leq 1$ \cite[Theorem 5]{B.Kowa}.
    \item If $f\in \mathcal{C}[A,B]$, where $-1 \leq B < A\leq 1 $ and $A - B \leq \lvert A^2 - 3 A B + 2 B^2 \rvert $, then $\det{T_{3,1}(f)}\leq 1$ \cite[Theorem 4]{Vkumar2}.
\end{enumerate}
\end{remark}
   %%%%%%%%%%%%%%%%%%%%%%%%%%%%%%%%%%%%%%%%%%%%%%%%%%%%%%%%%%%%%%%%%%%%%%%%%%%%%%%%%%%%%%%%%%%%%%%%%%
   %%%%%%%%%%%%%%%%%%%%%%%%%%%%%%%%%%%%%%%%%%%%%%%%%%%%%%%%%%%%%%%%%%%%%%%%%%%%%%%%%%%%%%%%%%%%%%%%%%%%%%%%%%%%%%

   %%%%%%%%%%%%%%%%%%%%%%%%%%%%%%%%%%%%%%%%%%%%%%%%%%%%%%%%%%%%%%%%%%%%%%%%%%%%%%%%%%%%%%%%%%%%%%%%%%%%%%%%%%%%%%%%%%%%%%%%%%%%%%%%%%%%%%%%%%%%%%%%%%%%%%
\begin{theorem}\label{Tlowerb}
     Let $f \in \mathcal{S}^*(\varphi) $ and $B_1^2 \geq B_2$, then the following inequalities hold:
\begin{enumerate}
     \item If $ \mu \notin  [0,4]$, then
\begin{equation}\label{ffff}
     \det{T_{3,1}(f)} \geq \min\left\{ 1-\frac{B_1^2}{4}, 1 - 2 B_1^2 + \frac{3 B_1^4}{4} + \frac{B_1^2 B_2}{2} - \frac{B_2^2}{4}\right\}.
\end{equation}
     \item If $  \mu = 4$, then
\begin{equation}\label{lowerb4}
   \det{T_{3,1}(f)}\geq 1 - 2 B_1^2 + \frac{3 B_1^4}{4} + \frac{B_1^2 B_2}{2} - \frac{B_2^2}{4}.
\end{equation}\label{lowerb5}
   \item If $\mu \in (0,4)$, then
\begin{equation}\label{lowerb3}
     \det{T_{3,1}(f)}\geq  1 - \frac{B_1^2}{4}-\frac{B_1^2 (B_1^2 + 3 B_1 - B_2 )^2}{4 \left( B_1 ( 2 B_1^2 - B_1 - 2 B_2 )+ ( 3 B_1^2 - B_2 ) (B_1^2 + B_2 ) \right)},
\end{equation}
\end{enumerate}
   where $$\mu = \frac{4 B_1 ( B_1^2 + 3 B_1 - B_2 )}{( 3 B_1^2 - B_2 ) ( B_1^2 +B_2 )+ B_1 ( 2 B_1^2 - 2 B_2 - B_1 ) }.$$
   The first two inequalities are sharp.
\end{theorem}
\begin{proof}
       Since $f(z)= z+ \sum_{n=2}^\infty a_n z^n \in \mathcal{S}^*(\varphi)$, therefore
\begin{equation}\label{expansion}
        \frac{zf'(z)}{f(z)}=\varphi( \omega(z)),
\end{equation}
       where $\omega(z)$ is a Schwarz function satisfying $\lvert \omega(z) \rvert \leq \lvert z \rvert$ and $\omega(0)=0$. Corresponding to the function $\omega$, suppose there is some  Carath\'{e}odory function $p(z)=1+ \sum_{n=1}^\infty p_n z^n \in \mathcal{P}$, which satisfy $\omega(z)=(p(z)-1)/(p(z)+1)$. On comparing the coefficients of $z$ and $z^2$ in (\ref{expansion}) with the series expansion of $f$, $\omega$ and $\varphi$, we obtain
      $$ a_2 =\frac{B_1 p_1}{2} \;\; \text{and}\;\; a_3 =\frac{1}{8}\left( (B_1^2 -B_1 +B_2 )p_1^2 +2 B_1 p_2\right).$$
      Since the class $\mathcal{S}^*(\varphi)$ and the class of  Carath\'{e}odory functions $\mathcal{P}$ are invariant under rotation and $\lvert p_1\rvert\leq 2$, without lose of generality, we can take $p_1\in [0,2]$.  Using the values of $a_2$, $a_3$ with Lemma \ref{lemma1}, we have
      $$ 2 \RE \left( a_2^2 \bar{a}_3 \right)=\frac{B_1^2 p_1^2}{16} \left(( B_1^2 -B_1 +B_2 )p_1^2+ B_1 \left(p_1^2 +(4-p_1^2) \RE{\bar{\zeta}} \right) \right) $$
      and
\begin{align*}
     -\lvert a_3\rvert^2 =& -\frac{1}{64}\bigg( ( B_1^2 - B_1 +B_2 )^2 p_1^4 +B_1^2 \left( p_1^4 + ( 4-p_1^2 )^2 \lvert \zeta\rvert^2 +2 p_1^2 (4-p_1^2) \RE{\bar{\zeta}}\right) \\
                       &+2 B_1 (B_1^2 -B_1 +B_2 ) p_1^2 \left( p_1^2 + (4-p_1^2 ) \RE{\bar{\zeta}} \right) \bigg).
\end{align*}
      Now, by (\ref{T31})
\begin{align*}
    \det{T_{3,1}(f)}=& -\frac{1}{64}\bigg( ( B_1^2 - B_1 +B_2 )^2 p_1^4 +B_1^2 \left( p_1^4 +( 4-p_1^2 )^2 \lvert \zeta\rvert^2 +2 p_1^2 (4-p_1^2 ) \RE{\bar{\zeta}}\right) \\
                &+2 B_1 (B_1^2 -B_1 +B_2 ) p_1^2 \left( p_1^2 + (4-p_1^2 ) \RE{\bar{\zeta}} \right) \bigg)-\frac{B_1^2 p_1^2 }{2}+1\\
                &+\frac{B_1^2 p_1^2}{16} \bigg(( B_1^2 -B_1 +B_2 )p_1^2+ B_1 \left(p_1^2 +(4-p_1^2 ) \RE{\bar{\zeta}} \right) \bigg).
\end{align*}
    We can rewrite the above equation as
\begin{align*}
    \det{T_{3,1}(f)}=&\left(\frac{3 B_1^4 +2 B_1^2 B_2 -B_2^2}{64}\right)p_1^4 -\frac{B_1^2}{2}p_1^2 -\frac{B_1^2}{64} (4-p_1^2 )^2 \lvert \zeta\rvert^2 \\                            &+\left(\frac{B_1^3- B_1 B_2}{32}\right) p_1^2 (4-p_1^2)\RE{\bar{\zeta}}+1 =:F(p_1^2, \lvert \zeta\rvert,\RE{\bar{\zeta}}).
\end{align*}
     Note that $F(p_1^2, \lvert \zeta\rvert,\RE{\bar{\zeta}})\geq F(p_1^2, \lvert \zeta\rvert,-\lvert \zeta\rvert)$, therefore
\begin{align*}
\det{T_{3,1}(f)}\geq &\left(\frac{3 B_1^4 +2 B_1^2 B_2 -B_2^2}{64}\right)p_1^4 -\frac{B_1^2}{2}p_1^2 -\frac{B_1^2}{64} \left(4-p_1^2 \right)^2 \lvert \zeta\rvert^2 \\                            & -\left(\frac{B_1^3- B_1 B_2}{32}\right) p_1^2 (4-p_1^2) \lvert \zeta\rvert+1 .
\end{align*}
     Let $x=p_1^2\in [0,4]$ and $y=\lvert \zeta\rvert\in [0,1]$, then $\det{T_{3,1}(f)} \geq  F(x,y)$, where
\begin{align*}
     F(x,y)=\left(\frac{3 B_1^4 +2 B_1^2 B_2 -B_2^2}{64}\right) x^2 -\frac{B_1^2}{2}x -\frac{B_1 (B_1^2 -B_2)}{32}&x \left(4-x \right)y  +1\\
                           &-\frac{B_1^2}{64}\left(4-x \right)^2 y^2.
\end{align*}
    If $B_1^2 \geq B_2$, then for any fixed $x$ and $y\in [0,1]$
     $$ \frac{\partial F}{\partial y}= -\frac{B_1 (B_1^2 -B_2)}{32} x (4-x) -\frac{B_1^2}{32} (4-x)^2 y   \leq 0, $$
   which means that $F(x,y)$ is decreasing function of $y$ and
   $$F(x,y)\geq F(x,1) =G(x)$$
   with
\begin{align*}
     G(x)=\frac{1}{64} \bigg( ( 3 B_1^2 - B_2 )( B_1^2 +B_2 )+ B_1 ( 2 B_1^2 - 2 B_2 -&B_1)\bigg) x^2 -\frac{B_1^2}{4} \\
                                                                &-\frac{1}{8} B_1 ( B_1^2 + 3 B_1  - B_2 ) x +1.
\end{align*}
    A computation shows that $G'(x)=0$ at
    $$x_0=\frac{4 B_1 ( B_1^2 + 3 B_1 - B_2 )}{( 3 B_1^2 - B_2 ) ( B_1^2 +B_2 )+ B_1 ( 2 B_1^2 - 2 B_2 - B_1 ) }$$
     and
     $$ G''(x_0)= \frac{1}{32} \left( ( 3 B_1^2 - B_2 )( B_1^2 +B_2)+ B_1 ( 2 B_1^2 - 2 B_2 -B_1 )\right) .$$
   Since $B_1 >0$ and $B_1^2 \geq B_2$,  numerator of $x_0$ is always positive. Note that, denominator of $x_0$ is same as $32 G''(x_0),$ which gives $x_0 <0$ $(\text{or}\; x_0>0)$ iff $G''(x_0)<0$ $(\text{or}\; G''(x_0)>0).$  Now, there arise two cases:\\
   %we and $x_0$'s  sign depends only on the numerator.
{\bf{Case 1:}} If $x_0 <0 $ or $x_0 >4 $, which means $G(x)$ does not have any critical point and
    \begin{align*}
    \det{T_{3,1}}(f)&\geq \min\left\{ G(0),G(4)\right\}\\
                    &= \min\left\{ 1-\frac{B_1^2}{4}, 1 - 2 B_1^2 + \frac{3 B_1^4}{4} + \frac{B_1^2 B_2}{2} - \frac{B_2^2}{4} \right\},
\end{align*}
   which yields \ref{ffff}. Both   $G(0)$ and $G(4)$ are sharp for the extremal functions $f_2$ and $f_1$ respectively, given by (\ref{extremalf1}).
      The case $x_0 <0$ also exhaust the possibility of $G''(x_0)<0.$\\
{\bf{Case 2:}} %If $x_0\in (0,4]$, now, we have $x_0 =0$ or $x_0 \in (0,4)$ or $x_0=4.$ Since $B_1>0$ and $B_1^2 \geq B_2$, we have $x_0 \neq 0.$ If $x_0=4$, then  $G''(x_0)>0$, in that case, we have
    If $x_0\in (0,4]$, then $G''(x_0)>0$ and the function $G$ attains its minimum value at $x_0$. Here, we discuss two possibilities for $x_0$, which are $x_0 = 4$ and $x_0\in (0,4)$.
     %First, when $x_0=0$.
      %Critical point $x_0$ is zero whenever  $B_1^2 +3 B_1 -B_2=0$. Therefore,
  %\begin{equation}\label{lowerb1}
%\begin{aligned}
 %  \det {T_{3,1}(f)}& \geq  G(0)   \\
  %                 &= 1-\frac{B_1^2}{4}.
%\end{aligned}
 % \end{equation}
   For $x_0 = 4$, we have
 \begin{equation}\label{lowerb2}
\begin{aligned}
     \det {T_{3,1}(f)}& \geq  G(4)   \\
                        &= 1 - 2 B_1^2 + \frac{3 B_1^4}{4} + \frac{B_1^2 B_2}{2} - \frac{B_2^2}{4},
\end{aligned}
\end{equation}
   which establishes (\ref{lowerb4}).
   If $x_0 \in (0,4)$, then %which is equivalent to
     %$0<  B_1 ( B_1^2 +3 B_1 - B_2 ) <    B_1 ( 2 B_1^2 - 2 B_2 -B_1  )+ (3 B_1^2 - B_2 ) ( B_1^2 + B_2 )$, we obtain
     $$\det{T_{3,1}(f)}\geq G(x_0),$$
     which proves (\ref{lowerb3}).
\end{proof}
%%%%%%%%%%%%           %%%%%%%%%%%%%%%%%%%%%%%%              %%%%%%%%%%%%%%%%%%%%%%%%%%%%%            %%%%%%%%%%%%%%%%%%%%%%%%%%%           %%%%%%%%%%%%%%%
%%%%%%%%%%%%%%%%%%%%%%%%%%%%%%%%%%%%%%%%%%%%%%%%%%%%%%%%%%%%%%%%%%%%%%%%%%%%%%%%%%%%%%%%%%%%%%%%%%%%%%%%%%%%%%%%%%%%%%%%%%%%%%%%%%%%%%%%%

%%%%%%%%%%%%%%%%%%%%%%%%%%%%%%%%%%%%%%%%%%%%%%%%%%%%%%%%%%%%
%%%%%%%%%%%%%%%%%%%%%%%%%%%%%%%%%%%%%%%%%%%%%%%%%%%%%%%%%%%%%%%%%%%%%%%%%%%%%%
    We get lower bound of $\det{T}_{3,1}(f)$ for different classes with corresponding alternatives of $\varphi(z)$.
\begin{corollary}
    Let $f \in \mathcal{S}^*[A,B]$, then the following hold:
\begin{enumerate}
  \item If $ \mu \notin [0,4]$, then
  $$ \det T_{3,1}(f) \geq \min \bigg\{ 1- \frac{(A - B)^2}{4},1 + \frac{(A - B)^2 ( 3 A^2 - 8 A B + 4 B^2 -8 )}{4}  \bigg\}.  $$
  \item If $ \mu = 4$, then
     $$ \det T_{3,1}(f) \geq 1 + \frac{(A - B)^2 ( 3 A^2 - 8 A B + 4 B^2 -8 )}{4}. $$
     \item If $\mu \in (0,4)$, then
     $$  \det T_{3,1}(f) \geq  1 - \frac{(A - B)^2 ( A^2 - 2 A B + B^2 + 2 A  + 2 )}{( + 3 A^2 - 8 A B + 4 B^2 + 2 A -1)},$$
\end{enumerate}
   where $ \mu =4 (3 + A)/( + 3 A^2 + 4 B^2 - 8 A B + 2 A -1 )$. The first two inequalities are sharp.
\end{corollary}
    In the following corollary, bounds are given for certain classes when coefficients $B_1$ and $B_2$ of $\varphi(z)$ satisfy the condition (\ref{ffff}) or (\ref{lowerb4}).
\begin{corollary}\label{crl1}
\begin{enumerate}
    \item If $f\in \mathcal{S}^*_{\varrho}$, then $\det{T_{3,1}(f)}\geq 0 $.
    \item If $f\in \mathcal{S}^*_{\sin}$, then $\det{T_{3,1}(f)}\geq -1/4$.
    \item If $f\in \mathcal{S}_P$, then $\det{T_{3,1}(f)}\geq 1 - 64 \left(19 \pi^4 -24 \pi^2 -432 \right)/(9 \pi^8)$.
    \item If $f\in \mathcal{S}^*_{RL}$, then $\det{T_{3,1}(f)}\geq -9 (4130 \sqrt{2}- 5861)/256$.
\end{enumerate}
\end{corollary}
%%%%%%%%%%%%%%%%%%%%%%%%%%%%%%%%%%%%%%%%%%%%%%%%%%%%%%%%%%%%%%%%%%%%%%%%%%%%%%%%%%%%%%%%%%%%%%%%%%%%%%%%%%%%%%%%%%%%%%%%%%%%%%%%%%%%%%%%%
\begin{remark} Some of the previously known results are established as special cases of the result, which are given below.
\begin{enumerate}
     \item If $f\in \mathcal{S}^*_{SG}$, then $\det{T_{3,1}(f)}\geq 35/64$ \cite[Theorem 2.1]{Cho}.
    \item If $f\in \mathcal{S}^*_{B}$, then $\det{T_{3,1}(f)}\geq 0$ \cite[Theorem 2.2]{Cho}.
    \item If $f\in \mathcal{S}^*_{Ne}$, then $\det{T_{3,1}(f)}\geq -1/4$ \cite[Theorem 4.2]{Sushil}.
    \item If $f\in \mathcal{S}^*_L$, then $\det{T_{3,1}(f)}\geq 135/256$ \cite[Theorem 3]{Vkumar}.
     \item If $f\in \mathcal{S}^*(1/2)$, then $\det{T_{3,1}(f)}\geq 0$ \cite[Corollary 4]{cudna}.
\end{enumerate}
\end{remark}
   Note that, the choice of $\varphi(z)$ satisfy the condition in (\ref{lowerb3}) for the classes $\mathcal{S}^*$ and $\Delta^*$ and by Theorem \ref{Tlowerb}, we have the following results.
\begin{remark}
\begin{enumerate}
  \item If $f\in \mathcal{S}^*$, then $\det{T_{3,1}(f)}\geq -1$ \cite[Corollary 3]{cudna},\cite[Corollary 2]{B.Kowa}.
  \item If $f\in \Delta^*$, then $\det{T_{3,1}(f)}\geq -1/15$ \cite[Theorem 2]{Vkumar}.
\end{enumerate}
\end{remark}
          %%%%%%%%%%%%%%%%%%%%%%%%%%%%%%%%%%%%%%%%%%%%%%%%%%%%%%%%%%%%%%%%%%%%%%%%%%%%%%%%%%%%%%%%%%%%%%%%%%%%%%%%%%%%%%%%%%%%%%%%%%%%%%%
\begin{theorem}\label{LowerHC}
    If $f\in \mathcal{C}(\varphi)$ such that $B_1^2 \geq 2 B_2$, then
\begin{align*}
   T_{3,1}(f) \geq &
\left\{
\begin{array}{lll}
    \min \bigg\{1 - \dfrac{B_1^2}{36}, 1 -\dfrac{B_1^2}{2} + \dfrac{B_1^4}{18} + \dfrac{B_1^2 B_2}{36} - \dfrac{B_2^2}{36}  \bigg\}, & \sigma \notin [0,4], \\  \\
     1 -\dfrac{B_1^2}{2} + \dfrac{B_1^4}{18} + \dfrac{B_1^2 B_2}{36} - \dfrac{B_2^2}{36}, & \sigma=4, \\ \\
     1 - \dfrac{B_1^3 (B_1^3 + 4 B_1^2 + 28 B_1 - 8 B_2)}{ 16 ( 2 B_1^4 + B_1^3 -B_1^2 - 2 B_1 B_2 + B_1^2 B_2 - B_2^2)}, &\sigma \in (0,4),
\end{array}
\right.
\end{align*}
  where $$\sigma = \frac{2 B_1 (B_1^2 + 16 B_1  - 2 B_2)}{ 2 B_1^4 + B_1^3 -B_1^2   - 2 B_1 B_2 + B_1^2 B_2 - B_2^2}.$$
  The first two inequalities are sharp.
\end{theorem}
\begin{proof}
   Let $f(z) = z+ \sum_{n=2}^\infty a_n z^n \in \mathcal{C}(\varphi)$, then we have
   $$ 1+ \frac{ z f''(z)}{f'(z)} = \phi(\omega(z)), $$
   where $\omega(z)$ is a Schwarz function. Corresponding to the function $\omega(z)$, suppose there is $p(z) = 1 + \sum_{n=2}^\infty p_n z^n \in \mathcal{P}$ such that $w(z) = (p(z)-1)/(p(z)+1)$. Comparison of same powers of $z$ in the above equation with the series expansions of $f(z)$, $\varphi(z)$ and $p(z)$ gives
   $$ a_2 = \frac{B_1 p_1 }{4} \quad \text{and} \quad a_3 = \frac{(  B_1^2 -B_1 + B_2 ) p_1^2 + 2 B_1 p_2}{24}. $$
  Following the same procedure as in Theorem \ref{Tlowerb} with Lemma \ref{lemma1}, we obtain $\det T_{3,1}(f) \geq F(x,y)$, where
%\begin{align*}
 %    T_{3,1}(f) =1 &- \frac{B_1^2 p_1^2}{8}+   \frac{1}{576}\bigg( {(2 B_1^4 + B_1^2 B_2 - B_2^2) p_1^4}  +  {(B_1^3 - 2 B_1 B_2) p_1^2 (4- p_1^2) \RE\bar{\zeta}}
  %       \\ &- {B_1^2 (4 - p_1^2)^2 \lvert \zeta \rvert^2}\bigg) =: F(p_1^2, \lvert \zeta\rvert , \RE\bar{\zeta})
%\end{align*}
 %  Since $ F(p_1^2, \lvert \zeta\rvert , \RE \bar{\zeta}) \geq F(p_1^2, \lvert \zeta\rvert , - \lvert {\zeta} \rvert) =: G(x , y)$, where
\begin{align*}
    F(x, y) =  1 &- \frac{B_1^2 x}{8} + \frac{1}{576}\bigg( (2 B_1^4 + B_1^2 B_2 - B_2^2) x^2  - (B_1^3 - 2 B_1 B_2) x (4- x) y \bigg) \\
                &- B_1^2 (4 - x)^2  y^2
\end{align*}
    for $x = p_1^2 \in [0,4]$ and $y = \lvert \zeta \rvert \in [0,1].$
    Partial derivative with respect to $y$ shows that
    $$ \frac{\partial F}{\partial y} = - \frac{(B_1^3 - 2 B_1 B_2) (4 - x) x}{576} -  \frac{B_1^2 (4 - x)^2 y}{288} \leq 0$$
    whenever $B_1^2 \geq 2 B_2$. Therefore $F$ is a decreasing function of $y$ and $F(x,y) \geq F(x,1)=:G(x),$ where
    $$ G(x) = 1 - \frac{B_1^2}{36} - \frac{B_1 ( B_1^2 + 16 B_1  - 2 B_2) x}{144} +  \frac{( 2 B_1^4 + B_1^3 -B_1^2  - 2 B_1 B_2 + B_1^2 B_2 - B_2^2)x^2}{576}. $$
    A simple computation reveals that $G'(x)=0$ at
    $$ x_0 = \frac{2 B_1 ( B_1^2 + 16 B_1  - 2 B_2)}{ 2 B_1^4 + B_1^3 -B_1^2 - 2 B_1 B_2 + B_1^2 B_2 - B_2^2} $$
    and
    $$ G''(x) = \frac{ 2 B_1^4  + B_1^3 -B_1^2  - 2 B_1 B_2 + B_1^2 B_2 - B_2^2 }{288}. $$
    Since $B_1^2 \geq B_2$ and $B_1 >0$, therefore numerator of $x_0$ is always positive.  Also note that, the denominator of $x_0$ and numerator of $G''(x)$ are same. Thus \\
{\bf{Case I:}} when $0 < x_0 < 4$, $G''(x_0) >0$ and hence minimum will attain at $x_0$.  In this case
\begin{align*}
   \det T_{3,1}(f) \geq \min G(x) &= G(x_0) \\
               &= 1 - \frac{B_1^3 (B_1^3 + 4 B_1^2 + 28 B_1 - 8 B_2)}{ 16 ( 2 B_1^4 + B_1^3 -B_1^2 - 2 B_1 B_2 + B_1^2 B_2 - B_2^2)}.
\end{align*}
{\bf{Case-II:}} If $x_0 < 0$ or $x_0 > 4$, which indicate that the critical point does not lie in the domain, then
\begin{align*}
   \det T_{3,1}(f) &\geq \min\{G(0), G(4)\}\\
              & = \min \bigg\{1 - \frac{B_1^2}{36}, 1 -\frac{B_1^2}{2} + \frac{B_1^4}{18} + \frac{B_1^2 B_2}{36} - \frac{B_2^2}{36}  \bigg\}.
\end{align*}
   Further, for $x_0 =4$, $ \det T_{3,1}(f) \geq G(4).$

   For the functions $f_3$ and $f_4$ defined in (\ref{ExtC}), we have
   $$\det T_{3,1}(f_3) = 1 - \frac{B_1^2}{36}, \quad \det T_{3,1}(f_4) =1 -\frac{B_1^2}{2} + \frac{B_1^4}{18} + \frac{B_1^2 B_2}{36} - \frac{B_2^2}{36} , $$
   which shows the sharpness of the bounds.
\end{proof}
    Theorem \ref{LowerHC} yields the following result for the class $\mathcal{C}$ and its subclasses.
\begin{corollary}
    Let $f \in \mathcal{C}[A,B]$ and $(A -B)^2  \geq 2 ( B^2 -A B )$, then the following bounds hold:
\begin{enumerate}
  \item If $\sigma \notin [0,4],$ then
  $$  T_{3,1}(f) \geq \min \bigg\{1 - \frac{(A-B)^2}{36}, 1 + \frac{(A - B)^2 (2 A^2 + 2 B^2  - 5 A B -18  )}{36}  \bigg\}.$$
  \item If $\sigma=4,$ then
  $$ T_{3,1}(f) \geq 1 + \frac{(A - B)^2 (2 A^2 + 2 B^2  - 5 A B -18  )}{36}. $$
  \item If $\sigma \in (0,4)$, then
  $$ T_{3,1}(f) \geq 1 - \frac{(A - B)^2 ( A^2 + B^2 + 4 A + 4 B - 2 A B + 28 )}{ 16 ( 2 A^2 + 2 B^2 + A  + B - 5 A B -1 )},$$
\end{enumerate}
   where $\sigma = 2 (A + B + 16 )/(2 A^2 + 2 B^2 + A  + B - 5 A B  -1 )$. The first two inequalities are sharp.
\end{corollary}
\begin{corollary}
   If $f \in \mathcal{C}$, then $\det T_{3,1}(f) \geq 0$. The bound is sharp.
\end{corollary}
%%%%%%%%%%%%%%%%%%%%%%%%%%%%%%%%%%%%%%%%%%%%%%%%%%%%%%%%%%%%%%%%%%%%%%%%%%%%%%%%%%%%%%%%%%%%%%%%%%%%%%%%%%%%%%%%%%%%%%%%%%%%%%%
    Recently, the lower and upper bounds of $\det{T}_{2,1}(f)$ and $\det{T}_{3,1}(f)$ for functions $f$ in $\mathcal{F}_1$, $\mathcal{F}_2$, $\mathcal{F}_3$ and $\mathcal{F}_4$  are obtained \cite{B.Kowa,Vkumar3,Lecko2}. By generalizing these works, when $f\in \mathcal{K}$, we obtain the bounds of $\det{T}_{2,1}(f)$ and $\det{T}_{3,1}(f)$  for different choices of $g\in \mathcal{S}^*$.
\begin{theorem}\label{KTM}
     Let $f\in \mathcal{K}$ for some function $g(z)=z+b_2 z^2+b_3 z^3+\cdots \in \mathcal{S}^*$, such that
\begin{equation}\label{hypo}
            6 \lvert b_2\rvert^3 - 4 \lvert b_2\rvert ( \lvert b_3 \rvert -1 ) - 4 (  \lvert b_3 \rvert-1 )^2 +
    \lvert b_2\rvert^2 ( 3  \lvert b_3\rvert +5  )\geq 0,
\end{equation}
   then
\begin{enumerate}
  \item $1-\left(1+\frac{\lvert b_2\rvert}{2}\right)^2\leq \det{T_{2,1}(f)}\leq 1.$
  \item If $ 6 \lvert b_2 \rvert^3 + \lvert b_2 \rvert^2 (3 \lvert b_3 \rvert +13)+4\lvert b_2 \rvert (\lvert b_3 \rvert -1) - 2 (1 - \lvert b_3 \rvert )^2  - 18 \leq 0$, then
\begin{equation}\label{uub1}
       \det{T_{3,1}(f)}\leq 1.
\end{equation}
  \item If  $ 6 \lvert b_2 \rvert^3 + \lvert b_2 \rvert^2 (3 \lvert b_3 \rvert +13)+4\lvert b_2 \rvert (\lvert b_3 \rvert -1) - 2 (1 - \lvert b_3 \rvert )^2  - 18 > 0$, then
\begin{equation}\label{uub}
       \det{T_{3,1}(f)} \geq \frac{1}{18} \bigg( 6 \lvert b_2 \rvert^3 + \lvert b_2 \rvert^2 (3 \lvert b_3 \rvert +13)+4\lvert b_2 \rvert (\lvert b_3 \rvert -1) - 2 (1 - \lvert b_3 \rvert )^2 \bigg).
\end{equation}
\end{enumerate}
   All these bounds are sharp.
\end{theorem}
\begin{proof}
    Let $f(z)= z+ \sum_{n=2}^\infty a_n z^n \in \mathcal{K}$, then for some $g\in \mathcal{S}^*$, we have
    $${z f'(z)}=  g(z) p(z).$$
    where $p(z)=1+p_1 z+p_2 z^2 +\cdots \in \mathcal{P}$. By comparing the coefficients of like powers on either side, we have
   $$a_2=\frac{b_2+p_1}{2} \;\; \text{and}\;\; a_3=\frac{1}{3}\left(b_3 +b_2 p_1+p_2\right).$$
   Using $\lvert p_n\rvert\leq 2$, we obtain
\begin{equation}\label{a2a3}
    \lvert a_2\rvert\leq \frac{2+ \lvert b_2 \rvert}{2} \;\; \text{and}\;\; \lvert a_3\rvert \leq \frac{1}{3}\left(\lvert b_3\rvert +2 \lvert b_2 \rvert +2\right).
\end{equation}
   When $g(z)= z+ b_2 z^2+b_3 z^3+\cdots $, the above bounds are sharp for the function $f_5$, given by
\begin{equation}\label{Extremal}
    f_5(z)=\int_0^z \frac{(1+t)g(t)}{t(1-t)}dt.
\end{equation}
   %Estimation for $\lvert a_2\rvert$  in (\ref{a2a3}) gives easily lower and upper bound of $\det{T}_{2,1}(f)$.
   From (\ref{a2a3}), we can easily obtain the lower and upper estimate of $\det{T}_{2,1}(f)$.
   The Lower bound of $\det{T_{2,1}(f)}$ is sharp for the function $f_5$ given by (\ref{Extremal}) %as for $f_3$, we have $\lvert a_2\rvert=(2+\lvert b_2\rvert)/2$,
    whereas the equality in the upper bound is attained for
   $$f_6(z)= \int_{0}^z \frac{(1+t^3)}{(1-t^3)}\frac{1}{(1-t)^2} dt =z+z^2+ z^3+\cdots. $$

    Using the inequality $\RE ( a_2^2 \bar{a}_3) \leq \lvert a_2 \rvert^2 \lvert a_3\rvert$ in (\ref{T31}), we get
    $$\det{T_{3,1}(f)}\leq \max{u(x)}, $$
    where $u(x)=2 \lvert a_2 \rvert^2 x- 2 \lvert a_2 \rvert^2 -x^2 +1 $ and $ x=\lvert a_3\rvert$.
     Note that $u(x)$ attains its maximum value at $x=\lvert a_2\rvert^2$. When $\lvert a_2 \rvert^2$ lies in the range of $x=\lvert a_3 \rvert$, that is  $\lvert a_2 \rvert^2 \leq \frac{1}{3}\left(\lvert b_3\rvert +2 \lvert b_2 \rvert +2\right)$, then
\begin{align*}
                     \max {u(x)}&= u(\lvert a_2\rvert^2)\\
                                   &=(\lvert a_2\rvert^2-1)^2\\
                             &\leq \left\{
        \begin{array}{ll}
              1,  &     \lvert a_2\rvert^2 \leq 2 , \\\\
         \left(  \frac{1}{3}\left(\lvert b_3\rvert +2 \lvert b_2 \rvert +2\right) -1\right)^2, &  2\leq \lvert a_2\rvert^2 \leq  \frac{1}{3} \left(2+2 \lvert b_2\rvert+\lvert b_3\rvert \right)
        \end{array}
        \right.\\
            &= \left\{
        \begin{array}{ll}
              1,  &     \frac{1}{3}\left(\lvert b_3\rvert +2 \lvert b_2 \rvert +2\right) \leq 2 , \\\\
         \left(  \frac{1}{3}\left(\lvert b_3\rvert +2 \lvert b_2 \rvert +2\right) -1\right)^2, &  2 \leq  \frac{1}{3} \left(2+2 \lvert b_2\rvert+\lvert b_3\rvert \right).
        \end{array}
        \right.
\end{align*}
    For $x = \lvert a_3 \rvert  \leq \lvert a_2 \rvert^2$,
\begin{equation}\label{max}
\begin{aligned}
   \max{u(x)}&= u(\frac{1}{3}\left(\lvert b_3\rvert +2 \lvert b_2 \rvert +2\right))\\
             &\leq \frac{1}{18} \left( 6 \lvert b_2 \rvert^3 + \lvert b_2 \rvert^2 (3 \lvert b_3 \rvert +13)+4\lvert b_2 \rvert (\lvert b_3 \rvert -1) - 2 (1 - \lvert b_3 \rvert )^2 \right).
\end{aligned}
\end{equation}
     Using (\ref{hypo}), we observe that the maximum value of $u(x)$ at $x=\lvert a_3 \rvert$, given in (\ref{max}),  is greater than the value at $x=\lvert a_2 \rvert^2$, that is $\left(  \frac{1}{3}\left(\lvert b_3\rvert +2 \lvert b_2 \rvert +2\right) -1\right)^2$. Moreover,
    $$ \frac{1}{18} \left( 6 \lvert b_2 \rvert^3 + \lvert b_2 \rvert^2 (3 \lvert b_3 \rvert +13)+4\lvert b_2 \rvert (\lvert b_3 \rvert -1) - 2 (1 - \lvert b_3 \rvert )^2 \right) > 1,$$
    whenever
\begin{equation}\label{condi}
      6 \lvert b_2 \rvert^3 + \lvert b_2 \rvert^2 (3 \lvert b_3 \rvert +13)+4\lvert b_2 \rvert (\lvert b_3 \rvert -1) - 2 (1 - \lvert b_3 \rvert )^2  - 18 > 0.
\end{equation}
    Upper bound of $\det{T_{3,1}(f)}$ can be obtained from (\ref{max}) and (\ref{condi}).

    Bound in (\ref{uub}) is sharp for the function $f_5$ given by (\ref{Extremal}) and for $g(z)=f(z)=z$ equality holds in (\ref{uub1}).
\end{proof}
   % For $g(z)=z$, class $\mathcal{K}$ reduces to the class $\mathcal{R}=\left\{ f\in \mathcal{S}: \RE f'(z)>0\right\}$ and Theorem \ref{KTM} gives the following result.
%\begin{corollary}
      %   If $f\in \mathcal{R}$, then
      %   $$ 0\leq \det{T_{2,1}(f)}\leq 1, \quad \det{T_{3,1}(f)}\leq 1.$$
%\end{corollary}
      If we replace $g(z)$ by $z/(1-z)$, $z/(1-z^2)$, $z/(1-z)^2$ and $z/(1-z+z^2)$, then we obtain the results for the classes  $\mathcal{F}_1$, $\mathcal{F}_2$, $\mathcal{F}_3$ and $\mathcal{F}_4$ respectively.
\begin{remark}
\begin{enumerate}
    \item If $f\in \mathcal{F}_1$, then $-5/4 \leq \det{T}_{2,1}(f)\leq 1$ \cite[Theorem 2.1]{Vkumar3}, \cite[Theorem 4]{Lecko2}.
    \item If $f\in \mathcal{F}_1$, then $T_{3,1}(f)\leq 11/9$ \cite[Theorem 2.2]{Vkumar3}, \cite[Theorem 5]{Lecko2}.
    \item If $f\in \mathcal{F}_2$, then $ 0 \leq  \det{T}_{2,1}(f)\leq 1$ \cite[Theorem 2.3]{Vkumar3},\cite[Theorem 2]{B.Kowa}.
    \item If $f\in \mathcal{F}_2$, then $T_{3,1}(f)\leq 1$ \cite[Theorem 2.4]{Vkumar3}, \cite[Theorem 3]{B.Kowa}.
    \item If $f\in \mathcal{F}_3$, then $ -3\leq  \det{T}_{2,1}(f)\leq 1$ \cite[Theorem 2.5]{Vkumar3},\cite[Theorem 5]{B.Kowa}.
    \item If $f\in \mathcal{F}_3$, then $T_{3,1}(f)\leq 8$ \cite[Theorem 2.6]{Vkumar3},\cite[Theorem 6]{B.Kowa}.
     \item If $f\in \mathcal{F}_4$, then $-5/4 \leq  \det{T}_{2,1}(f)\leq 1$ \cite[Theorem 2.7]{Vkumar3}, \cite[Theorem 2]{Lecko2}.
     \item If $f\in \mathcal{F}_4$, then $T_{3,1}(f)\leq 1$ \cite[Theorem 2.8]{Vkumar3}, \cite[Theorem 3]{Lecko2}.
\end{enumerate}
\end{remark}

\begin{theorem}\label{lastt}
     Let $f\in \mathcal{K}$ and $g(z)=z+ \sum_{n=2}^\infty b_n z^n \in \mathcal{S}^*$. If $\Tilde{g}(z)=z+ \sum_{n=2}^\infty i^{n-1} b_n z^n \in \mathcal{S}^*$, then the following sharp bound hold:
     $$ \lvert \det{T_{2,2}(f)}\rvert\leq \frac{1}{4} \left(2+\lvert b_2\rvert \right)^2+\frac{1}{9}\left( \lvert b_3\rvert +2 \lvert b_2\rvert+2\right)^2.$$
\end{theorem}
\begin{proof} From (\ref{intil}), we have
    $$\lvert \det{T_{2,2}(f)}\rvert = \lvert  a_2^2 -\lvert a_3\rvert^2\rvert\leq   \lvert a_3\rvert^2 + \lvert  a_2 \rvert^2.$$
    % If $f\in \mathcal{K}$, then we have
     %$$\lvert \det{T_{2,2}(f)}\rvert = \lvert  a_2^2 -\lvert a_3\rvert^2\rvert  \leq   \lvert a_3\rvert^2 + \lvert  a_2 \rvert^2.$$
     Using the estimate of $\lvert a_2\rvert$ and $\lvert a_3\rvert$ for $f\in \mathcal{K}$, given in $(\ref{a2a3})$, we obtain the required bound of $\lvert \det{T_{2,2}(f)}\rvert$.

     The bound is sharp for the function $f_7$ given by
     $$\frac{z f_7'(z)}{\Tilde{g}(z)}=\frac{1+i z}{1- iz}$$
     as $f_7$ has the following power series representation
     $$f_7(z)=z + i\left( 2  + b_2\right) \frac{z^2}{2} - \left( 2 + b_3 + 2 b_2 \right) \frac{z^3}{3} +\cdots.$$
\end{proof}
     By changing the function $g$ in Theorem \ref{lastt}, we can obtain results for other classes.
\begin{corollary}
\begin{enumerate}
    \item If $f\in \mathcal{F}_1$, then  $\lvert \det{T_{2,2}(f)} \rvert \leq 181/36$ and the bound is sharp for
     $$ f(z)= \int_0^{z} \left(\frac{1+i t}{1- i t}\right) \frac{\Tilde{g}(t)}{t} dt,$$
    where $\Tilde{g}(z)=z/(1- iz)\in \mathcal{S}^*.$
    \item If $f\in \mathcal{F}_2$, then $\lvert \det{T_{2,2}(f)} \rvert \leq 2$ and the bound is sharp for
     $$ f(z)= \int_0^{z} \left(\frac{1+i t}{1- i t}\right) \frac{\Tilde{g}(t)}{t} dt,$$
    where $\Tilde{g}(z)=z/(1 + z^2)\in \mathcal{S}^*.$
    \item If $f\in \mathcal{F}_3$, then $\lvert \det{T_{2,2}(f)} \rvert \leq 13 $ and the bound is sharp for
     $$ f(z)= \int_0^{z} \left(\frac{1+i t}{1- i t}\right) \frac{\Tilde{g}(t)}{t} dt,$$
    where $\Tilde{g}(z)=z/(1- i z)^2\in \mathcal{S}^*.$
    \item If $f\in \mathcal{F}_4$, then $\lvert \det{T_{2,2}(f)} \rvert \leq  145/36 $ and the bound is sharp for
     $$ f(z)= \int_0^{z} \left(\frac{1+i t}{1- i t}\right) \frac{\Tilde{g}(t)}{t} dt,$$
    where $\Tilde{g}(z)=z/(1- iz + z^2)\in \mathcal{S}^*.$
    \item If $f\in \mathcal{R}$, then $\lvert \det{T_{2,2}(f)} \rvert \leq  13/9$ and the bound is sharp for
       $$f(z)= \int_0^{z} \left(\frac{1+i t}{1- i t}\right)  dt.$$
\end{enumerate}
\end{corollary}
%%%%%%%%%%%%%%%%%%%%%%%%%%%%%%%%%%%%%%%%%%%%%%%%%%%%%%%%%%%%%%%%%%%%%%%%%%%%%%%%%%%%%%%%%%%%%%%%%%%%%%%%%%%%%%%%%%%%%%%%%%
\section*{Declarations}
\subsection*{Funding}
The work of the Surya Giri is supported by University Grant Commission, New-Delhi, India  under UGC-Ref. No. 1112/(CSIR-UGC NET JUNE 2019).
\subsection*{Conflict of interest}
	The authors declare that they have no conflict of interest.
\subsection*{Author Contribution}
    Each author contributed equally to the research and preparation of manuscript.
\subsection*{Data Availability} Not Applicable.
%%%%%%%%%%%%%%%%%%%%%%%%%%%%%%%%%%%%%%%%%%%%%%%%%%%%%%%%%%
\noindent
%%%%%%%%%%%%%%%%%%%%%%%%%%%%%%%%%%%%%%%%%%%%%%%%%%%%%%%%%%%%%%%%%%%%%%%%%%%%%%%%%%%%%%%%%%%%%%%%%%%

\end{document}